\documentclass[a4paper]{amsart}

\usepackage{amsthm, amssymb}
\usepackage{cite}
\usepackage[dvipdfmx]{graphicx}
\usepackage{algorithm, algorithmic}

\theoremstyle{definition}
\newtheorem{theorem}{Theorem}[section]
\newtheorem{lemma}[theorem]{Lemma}

\newtheorem{observation}[theorem]{Observation}

\theoremstyle{definition}
\newtheorem{definition}[theorem]{Definition}

\theoremstyle{remark}
\newtheorem{remark}[theorem]{Remark}

\newenvironment{rmenum}{
\begin{enumerate}

}
{\end{enumerate}}

\newcommand{\parNei}[2]{N_{#1}(#2)}
\newcommand{\parcut}[2]{\delta_{#1}(#2)}

\newcommand{\acset}[2]{\mathcal{S}^{#2}(#1)}

\newcommand{\nacset}[2]{\mathcal{T}^{#2}(#1)}
\newcommand{\sqrset}[1]{\mathcal{A}(#1)}

\newcommand{\reachset}[4]{\Lambda_{#1, #2}(#3, #4)}
\newcommand{\reachsets}[3]{\Lambda_{#1, #2}(#3)}
\newcommand{\oNei}[3]{{}_{#3} N_{#1}(#2)}
\newcommand{\ocut}[3]{{}_{#3}\delta_{#1}(#2)}

\newcommand{\rrset}[2]{\mathcal{R}^{#2}(#1)}
\newcommand{\lsrset}[2]{\mathcal{L}^{#2}(#1)}

\title[Bidirected Critical Graphs II]{Constructive Characterization for\\ Signed Analogue of Critical Graphs II:\\ General Radials and Semiradials}
\author{Nanao Kita}
\address{Tokyo University of Science 2641 Yamazaki, Noda, Chiba, Japan 278-0022}
\email{kita@rs.tus.ac.jp}
\date{\today}

\begin{document}

\begin{abstract}
This paper is a sequel of our preceding paper (N. Kita:  Constructive characterization for signed analogue of critical graphs I: Principal classes of radials and semiradials. arXiv preprint, arXiv:2001.00083, 2019). 
In the preceding paper, 
 the concepts of radials and semiradials are introduced, 
and   constructive characterizations for five principal classes of radials and semiradials are provided.  
Radials are a common analogue of critical graphs from matching theory and a class of directed graphs called flowgraphs, 
whereas semiradials are a relaxed concept of radials.  
Five classes of radials and semiradials, 
that is, 
absolute semiradials, strong and almost strong radials, linear semiradials, and sublinear radials, 
were defined and characterized in the paper. 
In this paper, 
we use these characterizations to provide a constructive characterization of general radials and semiradials. 
\end{abstract} 

\maketitle

\section{Definitions} 

We use the notation introduced in our preceding paper~\cite{kita2019constructive}. 
For basic notations for sets and graphs, we mostly follow Schrijver~\cite{schrijver2003}. 
In the following, we list  some newly introduced notation. 
Let $\alpha \in \{+, -\}$. Let $G$ be a bidirected graph, and let $S \subseteq V(G)$. 
We denote by $\ocut{G}{S}{\alpha}$ the set of edges from $\parcut{G}{S}$ in which the end in $S$ has the sign $\alpha$.  
The set $\oNei{G}{S}{\alpha}$ denotes the set of vertices from $\parNei{G}{S}$ 
that are joined with vertices in $S$ by an edge from $\ocut{G}{S}{\alpha}$. 
Let $X\subseteq V(G)$. 
The bidirected graph obtained from $G$ by contracting $X$ is denoted by $G/X$. 
That is, $G[X]$ is the bidirected graph obtained from $G$ by regarding $X$ as a single vertex $[X]$ and deleting the arising loops over $[X]$.

\begin{definition} 
Let $\alpha, \beta \in\{+, -\}$, let $G$ be a bidirected graph, and let $r\in V(G)$.  
The set of vertices that can reach $r$ by an $(\alpha, \beta)$-ditrail is denoted by $\reachset{G}{r}{\alpha}{\beta}$. 
The set of vertices that can reach $r$ by an $\alpha$-ditrail is denoted by $\reachsets{G}{r}{\alpha}$. 
\end{definition}

\begin{definition} 
Let $\alpha\in\{+, -\}$, let $G$ be a bidirected graph, and let $r\in V(G)$. 
A vertex $x\in V(G)$ is {\em absolute} with respect to $r$ 
if $x \in \reachsets{G}{r}{\alpha} \cap \reachsets{G}{r}{-\alpha}$ holds. 
By contrast, $x$ is {\em $\alpha$-linear} with respect to $r$ 
if $x \in \reachsets{G}{r}{\alpha}\setminus \reachsets{G}{r}{-\alpha}$ holds. 
A vertex $x\in V(G)$ is {\em $\alpha$-strong}  with respect to $r$ if 
$x \in \reachset{G}{r}{\alpha}{-\alpha}\cap \reachset{G}{r}{-\alpha}{-\alpha}$ holds. 
By contrast, $x$ is {\em $\alpha$-sublinear}  with respect to $r$ if 
$x\in \reachset{G}{r}{\alpha}{-\alpha} \setminus \reachset{G}{r}{-\alpha}{-\alpha}$ holds.  
If $G$ is an $\alpha$-semiradial with root $r$, 
then an absolute or $\alpha$-linear vertex with respect to $r$ can simply be called an absolute or linear vertex, respectively. 
If $G$ is an $\alpha$-radial with root $r$, 
then an $\alpha$-strong  or sublinear vertex with respect to $r$ can simply be called a strong or sublinear vertex, respectively. 
\end{definition}

\begin{definition} 
Let $G$ and $H$ be disjoint bidirected graphs. 
Let $s \in V(G)$ and $T\subseteq V(H)$. 
A gluing sum of $G$ and $H$ with respect to $s$ and $T$ 
is a bidirected graph obtained by identifying $s$ and $T$. 
More precisely, 
a bidirected graph $\hat{G}$ is a {\em gluing sum} of $G$ and $H$ with respect to $s$ and $T$  
if it satisfies the following properties: 
\begin{rmenum} 
\item $V(\hat{G}) =  ( V(G)\setminus \{s\} ) \cup V(H)$.  
\item Let $F \subseteq E(G)$ be the union of $\parcut{G}{s}$ and the set of loops over $s$. 
Then, $( E(G)\setminus F)  \cup E(H) \subseteq E(\hat{G})$. 
Also, there is a one-to-one mapping $f$ from $F$ to $E(\hat{G})\setminus (E(G)\setminus F) \setminus E(H)$ 
that satisfies the following properties: 
\begin{rmenum} 
\item 
If $e \in F$ is an edge between $s$ and $v\in \parNei{G}{s}$, 
then $f(e)$ is an edge between a vertex $t \in T$ and $v$  
 such that the signs of $t$ and $v$ over $f(e)$ are equal to the signs of $s$ and $v$ over $e$, respectively; and, 
\item  if $e\in F$ is a loop, then $f(e)$ is an $(\alpha_1, \alpha_2)$-edge between (possibly identical) vertices from $T$, 
where $\alpha_1, \alpha_2 \in \{+ , -\}$ are the signs of $s$ over $e$. 
\end{rmenum} 
\end{rmenum} 
We denote such $\hat{G}$ by $(G; s)\oplus (H; T)$.  
\end{definition}

\begin{remark} 
Note that a gluing sum is not always uniquely determined. 
However, for convenience, 
we may write $\hat{G} = (G; s)\oplus (H; T)$ 
to claim that $\hat{G}$ is a gluing sum of $G$ and $H$ with respect to $s$ and $T$.  
\end{remark}

\section{Edge Adding and Deleting Lemmas}  \label{sec:edgelem} 

\subsection{Edge Adding Lemmas} \label{sec:edgelem:add} 

In Section~\ref{sec:edgelem:add}, we introduce Lemmas~\ref{lem:edgeaddr} and \ref{lem:edgeaddsr}  to be used in later sections. 
For proving Lemma~\ref{lem:edgeaddr}, 
we first provide and prove Lemmas~\ref{lem:nobypass} and \ref{lem:oneadd}.

\begin{lemma} \label{lem:nobypass} 
Let $\alpha \in \{+, -\}$. 
Let $G$ be a bidirected graph, and let $r\in V(G)$. 
Let $x, y\in V(G)$. 
If either $x, y \in \reachset{G}{r}{\alpha}{-\alpha} \setminus \reachset{G}{r}{-\alpha}{-\alpha}$ 
or $x, y \in \reachsets{G}{r}{\alpha} \setminus \reachsets{G}{r}{-\alpha}$ holds, then there is no $(-\alpha, -\alpha)$-ditrail from $x$ to $y$. 
\end{lemma} 
\begin{proof} 
Suppose the contrary, and let $P$ be a $(-\alpha, -\alpha)$-ditrail from $x$ to $y$. 
First, consider the case $x, y \in \reachsets{G}{r}{\alpha} \setminus \reachsets{G}{r}{-\alpha}$. 
Let $Q$ be an $\alpha$-ditrail from $y$ to $r$. 
Trace $Q^{-1}$ from $r$, and let $z$ be the first encountered vertex from $P$. 
Then, either $xPz + zQr$ or $yP^{-1}z + zQr$ is a $-\alpha$-ditrail 
that contradicts $x\not\in \reachsets{G}{r}{-\alpha}$ or $y\not\in \reachsets{G}{r}{-\alpha}$, respectively. 
The case $x, y \in \reachset{G}{r}{\alpha}{-\alpha} \setminus \reachset{G}{r}{-\alpha}{-\alpha}$ can be proved in a similar way. 
The lemma is proved.

\end{proof}

Lemma~\ref{lem:nobypass} implies Lemma~\ref{lem:oneadd}.

\begin{lemma} \label{lem:oneadd} 
Let $\alpha \in \{+, -\}$. 
Let $G$ be a bidirected graph, 
and let $H$ be an induced subgraph of $G$  that is an $\alpha$-radial with root $r \in V(H)$. 
Assume $\parNei{G}{V(G)\setminus V(H)} \cap \reachset{G}{r}{-\alpha}{-\alpha} = \emptyset$. 
Let $u \in V(H)\setminus \reachset{G}{r}{-\alpha}{-\alpha}$ and $v\in V(G)\setminus V(H)$. 
Let $G'$ be a bidirected graph obtained from $G$ by adding an edge $e$ between $u$ and $v$ in which the sign of $u$ is $\alpha$. 
Let $x\in V(G)$ and $\beta\in\{+, -\}$.  
If $G'$ has a $(\beta, -\alpha)$-ditrail $P$ from $x$ to $r$ with $e\in E(P)$, 
then $G'$ has a $(\beta, -\alpha)$-ditrail $P'$ from $x$ to $r$ with $E(P') \cap \parcut{G}{H} \subseteq E(P)\cap \parcut{G}{H} \setminus \{e\}$. 
\end{lemma} 
\begin{proof} 
Let $P$ be a $(\beta, -\alpha)$-ditrail $P$ from $x$ to $r$ with $e\in E(P)$. 
If $P$ contains $(v, e, u)$ as a subditrail, then $uPr$ is a $(-\alpha, -\alpha)$-ditrail from $u$ to $r$, which contradicts $u\not\in \reachset{G}{r}{-\alpha}{-\alpha}$.  
Therefore, $P$ contains $(u, e, v)$. Thus, $xPu$ is a $(\beta, -\alpha)$-ditrail.

First, consider the case $x\in V(G)\setminus V(H)$. 
Trace $P$ from $x$, and let $y$ be the first encountered vertex that is in $V(H)$. 
Note $e\not\in E(xPy)$. 
Note also that  $xPy$ is a $(\beta, -\alpha)$-ditrail; 
for, otherwise, $yPu$ is a $(-\alpha, -\alpha)$-ditrail, which contradicts Lemma~\ref{lem:nobypass}. 
Let $Q$ be an $(\alpha, -\alpha)$-ditrail of $H$ from $y$ to $r$. 
Then, $xPy + Q$ is a desired ditrail $P'$. 

Next, consider the case $x\in V(H)$. 
Let $\beta = \alpha$. Because $H$ is an $\alpha$-radial with root $r$,  
there is clearly an $(\alpha, -\alpha)$-ditrail $P'$ in $G'$ that satisfies the condition. 
Now, let $\beta = -\alpha$.  Then,  $xPu$ is a $(-\alpha, -\alpha)$-ditrail. 
Let $R$ be an $(\alpha, -\alpha)$-ditrail of $H$ from $x$ to $r$. 
Trace $R^{-1}$ from $r$, and let $z$ be the first encountered vertex from $xPu$.  
Because there is no $(-\alpha, -\alpha)$-ditrail between $u$ and $r$, 
it follows that $xPz + zRr$ is a $(-\alpha, -\alpha)$-ditrail from $x$ to $r$. 
This ditrail is a desired ditrail $P'$. 
Thus, the claim is proved for the case $x\in V(H)$. 
This proves the lemma. 
\end{proof}

Lemma~\ref{lem:edgeaddr} can be derived from Lemma~\ref{lem:oneadd}.

\begin{lemma} \label{lem:edgeaddr} 
Let $\alpha \in \{+, -\}$, let $G$ be a bidirected graph, 
and let $H$ be an induced subgraph of $G$ that is an $\alpha$-radial with root $r \in V(H)$. 
 Assume $\parNei{G}{V(G)\setminus V(H)} \cap \reachset{G}{r}{-\alpha}{-\alpha} = \emptyset$. 
Let $G'$ be a bidirected graph obtained from $G$ by adding some edges between $V(H) \setminus \reachset{G}{r}{-\alpha}{-\alpha}$ and $V(G)\setminus V(H)$ 
in which the ends in $V(H)$ have the sign $\alpha$. 
Then, $\reachset{G'}{r}{\beta}{-\alpha} = \reachset{G}{r}{\beta}{-\alpha}$ holds for each $\beta\in\{+, -\}$.  
\end{lemma} 
\begin{proof} 
Let $F := E(G')\setminus E(G)$. That is, $G' = G + F$. 
We prove the lemma by induction on $|F|$. 
If $|F| = 1$, then Lemma~\ref{lem:oneadd} proves the claim. 
Let $|F| > 1$, and assume that the claim holds for every case where $|F|$ is smaller. 
Let $e\in F$, and let $G'' := G + F\setminus \{e\}$.  Hence, $G' = G'' + e$. 
From the induction hypothesis, 
$\reachset{G''}{r}{\beta}{-\alpha} = \reachset{G}{r}{\beta}{-\alpha}$ holds for each $\beta\in\{+, -\}$.   
Thus, $H$ is an induced subgraph of $G''$ that is an $\alpha$-radial with root $r \in V(H)$ 
such that $\parNei{G''}{V(G'') \setminus V(H)} \cap \reachset{G''}{r}{-\alpha}{-\alpha} = \emptyset$, 
and the end of $e$ from $V(H)$ is disjoint from $\reachset{G}{r}{-\alpha}{-\alpha}$. 
Therefore, Lemma~\ref{lem:oneadd} implies that 
$\reachset{G'}{r}{\beta}{-\alpha} = \reachset{G''}{r}{\beta}{-\alpha}$ holds for each $\beta\in\{+, -\}$.    
Consequently, $\reachset{G'}{r}{\beta}{-\alpha} = \reachset{G}{r}{\beta}{-\alpha}$ is obtained for each $\beta\in\{+, -\}$.  
This proves the lemma. 
\end{proof}

It can easily be confirmed from  similar discussions that 
the semiradial versions of Lemmas~\ref{lem:oneadd} and \ref{lem:edgeaddr} also hold. 
As such, Lemma~\ref{lem:edgeaddsr} is obtained.

\begin{lemma}  \label{lem:edgeaddsr} 
Let $\alpha \in \{+, -\}$, let $G$ be a bidirected graph, 
and let $H$ be an induced subgraph of $G$ that is an $\alpha$-semiradial with root $r \in V(H)$.  
Assume $\parNei{G}{V(G)\setminus V(H)} \cap \reachsets{G}{r}{-\alpha} = \emptyset$. 
Let $G'$ be a bidirected graph obtained from $G$ by adding some edges between $V(H) \setminus \reachsets{G}{r}{-\alpha}$ and $V(G)\setminus V(H)$ 
in which the ends in $V(H)$ have the sign $\alpha$. 
Then, $\reachsets{G'}{r}{\beta} = \reachsets{G}{r}{\beta}$ holds for each $\beta\in\{+, -\}$.  
\end{lemma}

\subsection{Edge Deleting Lemmas}  \label{sec:edgelem:del} 

In Section~\ref{sec:edgelem:del}, we provide Lemmas~\ref{lem:r2delete} and \ref{lem:sr2delete} to be used in later sections. 
Lemma~\ref{lem:edgeaddr} easily implies Lemma~\ref{lem:r2delete}.

\begin{lemma} \label{lem:r2delete} 
Let $\alpha \in \{+, -\}$, let $G$ be a bidirected graph, 
and let $H$ be an induced subgraph of $G$ that is an $\alpha$-radial with root $r \in V(H)$ 
such that $\parNei{G}{V(G)\setminus V(H)} \cap \reachset{G}{r}{-\alpha}{-\alpha} = \emptyset$. 
Let $F \subseteq \ocut{G}{H}{\alpha}$. 
Then,  $\reachset{G - F}{r}{\beta}{-\alpha} = \reachset{G}{r}{\beta}{-\alpha}$ holds for each $\beta\in\{+, -\}$.  
\end{lemma} 
\begin{proof} 
Let $G' := G - F$. That is, $G = G' + F$.  
It is clear that $H$ is an induced subgraph of $G'$  that is an $\alpha$-radial with root $r$. 
Because $\reachset{G'}{r}{-\alpha}{-\alpha}  \subseteq \reachset{G}{r}{-\alpha}{-\alpha}$ clearly holds,  
we have $\parNei{G'}{V(G')\setminus V(H)} \cap \reachset{G'}{r}{-\alpha}{-\alpha} = \emptyset$.   
Hence, Lemma~\ref{lem:edgeaddr} implies the claim. 
\end{proof}

It is easily observed that the semiradial version of Lemma~\ref{lem:r2delete} also holds,  
which is stated as Lemma~\ref{lem:sr2delete}.

\begin{lemma} \label{lem:sr2delete} 
Let $\alpha \in \{+, -\}$, let $G$ be a bidirected graph,  
and let $H$ be an induced subgraph of $G$ that is an $\alpha$-semiradial with root $r$ 
such that $\parNei{G}{V(G)\setminus V(H)} \cap \reachsets{G}{r}{-\alpha} = \emptyset$. 
 Let $F \subseteq \ocut{G}{H}{\alpha}$. 
Then,  $\reachsets{G -F}{r}{\beta} = \reachsets{G}{r}{\beta}$ holds for each $\beta\in\{+, -\}$.  
\end{lemma}

\section{Neighbor Lemma}

\begin{lemma} \label{lem:neigh2ear} 
Let $G$ be a bidirected graph, let $r\in V(G)$, 
and let $S$ be a set of vertices with $\{r\} \subseteq S \subsetneq V(G)$. 
Let $x\in V(G)\setminus S$ and $y \in S$ be adjacent vertices,  
and let $\beta \in \{+, -\}$ be the sign of $x$ over $xy$.  
If $G$ has a $-\beta$-ditrail from $x$ to $r$,  then there is a diear relative to $S$  that is either 
\begin{rmenum} 
\item a simple diear that starts with subditrail $(y, yx, x)$ or ends with subditrail $(x, xy, y)$, or 
\item a scoop diear whose grip is $xy$. 
\end{rmenum} 
\end{lemma} 
\begin{proof} 
Let $P$ be a $-\beta$-ditrail from $x$ to $r$. 
Trace $P$ from $x$, and let $z$ be the first encountered vertex in $S$. 
If $xPz$ does not contain $xy$, then $(y, yx, x) + xPz$ is a simple diear relative to $S$. 
If $xPz$ contains $xy$, then $xPz$ ends with the sequence $(x, xy, y)$; 
this further implies that $(y, yx, x) + xPz$ forms a scoop diear relative to $S$ whose grip is $xy$. 
\end{proof}

\section{Grounds for Semiradials}  

In this section, we introduce the concepts of absolute and linear grounds for semiradials 
and provide some properties to be used in later sections. 
Lemma~\ref{lem:sr2union} can easily be confirmed. 

\begin{lemma}\label{lem:sr2union}  
Let $\alpha\in\{+, -\}$. 
Let $G$ be a bidirected graph, and let $r\in V(G)$. 
Let $H_1$ and $H_2$ be subgraphs of $G$ that are $\alpha$-semiradials with root $r$.  
Then, $H_1 + H_2$ is also an $\alpha$-semiradial with root $r$. 
Furthermore, 
\begin{rmenum} 
\item if $H_1$ and $H_2$ are both absolute, then $H_1 + H_2$ is also absolute; and, 
\item if all vertices of $H_1$ and $H_2$ except $r$ are $\alpha$-linear in $G$ with respect to $r$,  which implies $H_1$ and $H_2$ are both linear, 
then $H_1 + H_2$ is also linear. 
\end{rmenum} 
\end{lemma}

Under Lemma~\ref{lem:sr2union}, absolute and linear grounds for semiradials can be defined.

\begin{definition} 
Let $\alpha\in\{+, -\}$. 
Let $G$ be a bidirected graph, and let $r\in V(G)$.  
The maximum  subgraph of $G$ that is an absolute $\alpha$-semiradial with root $r$ is called the {\em absolute ground}.   
By contrast, 
the {\em linear ground} of $G$ 
is the maximum  subgraph that is a linear $\alpha$-semiradial with root $r$  
in which every vertex except the root is linear in $G$. 
\end{definition}

Every $\alpha$-semiradial has absolute and linear grounds 
because $r$ induces a subgraph that is trivially an absolute or linear $\alpha$-semiradial. 
An absolute or linear ground is said to be {\em trivial} if it comprises only one vertex, that is, the root, and no edge.  
The intersection of the absolute and linear grounds comprises only one vertex, that is, the root.

\begin{lemma} \label{lem:rnei2alt} 
Let $\alpha\in\{+, -\}$. 
Let $G$ be an $\alpha$-semiradial with root $r\in V(G)$. 
Then, every neighbor of $r$ is contained in either the absolute or linear grounds of $G$. 
If $v\in \parNei{G}{r}$ is linear in $G$, then it is contained in the linear ground of $G$. 
Otherwise, that is, if $v$ is not linear in $G$, then it is contained in the absolute ground of $G$. 
\end{lemma}

Absolute ground of an $\alpha$-semiradial is trivial if and only if every neighbor of $r$ is linear. 
Linear ground of an $\alpha$-semiradial is trivial if and only if no neighbor of $r$ is linear. 

\begin{definition} 
Let $\alpha\in\{+, -\}$. 
We say that an $\alpha$-semiradial with root $r$ is {\em sharp}  if 
every neighbor of $r$ is linear. 
\end{definition} 

That is, by Lemma~\ref{lem:rnei2alt},  
an $\alpha$-semiradial with root $r$ is sharp if and only if its absolute ground is trivial. 
We use this fact everywhere in this paper sometimes without explicitly mentioning it.

\begin{lemma} \label{lem:abgnei2lin} 
Let $\alpha\in\{+, -\}$. 
Let $G$ be an $\alpha$-semiradial with root $r\in V(G)$. 
Let $H$ be the absolute ground of $G$. 
Then, every neighbor of $H$ is linear in $G$. 
\end{lemma} 
\begin{proof}  
Suppose that $x\in\parNei{G}{H}$ is absolute in $G$. 
Then, Lemma~\ref{lem:neigh2ear}  implies that there is a diear relative to $H$.  
By Theorem~\ref{thm:asr},  this contradicts the maximality of $H$. 
Thus, the lemma is proved. 
\end{proof}

\begin{lemma} \label{lem:lg2neigh}
Let $\alpha\in\{+, -\}$, and let $G$ be a sharp $\alpha$-semiradial with root $r\in V(G)$.  
Let $H$ be the linear ground of $G$. 
If $V(G)\setminus V(H) \neq \emptyset$, then the following hold:  
\begin{rmenum} 
\item \label{item:noear}
There is no simple $(-\alpha, -\alpha)$-diear relative to $H$. 
\item \label{item:scoop} 
$\ocut{G}{H}{-\alpha}\neq \emptyset$.  
For each $e \in \ocut{G}{H}{-\alpha}$,  the end $u \in V(G)\setminus V(H)$ of $e$ is absolute in $G$, and 
there is a $-\alpha$-scoop diear relative to $H$ whose grip is $e$. 

\end{rmenum} 
\end{lemma} 
\begin{proof} 
For proving \ref{item:noear}, 
suppose the contrary, and let $P$ be a $(-\alpha, -\alpha)$-ditrail relative to $H$. 
Let $y$ and $z$ be the first and last vertices of $P$, respectively. 
Let $Q$ be an $\alpha$-ditrail of $H$ from $z$ to $r$. 
Then, $P + Q$ is an $-\alpha$-ditrail from $y$ to $r$, which contradicts $y\in V(H)$. 
Therefore, \ref{item:noear} follows.

Let $x\in V(G)\setminus V(H)$. 
Then, it is easily observed that 
there is an $(\alpha, -\alpha)$-ditrail from $x$ to a vertex in $H$ whose edges are disjoint from $E(H)$.  
This implies $\ocut{G}{H}{-\alpha} \neq \emptyset$.   

If $u$ is linear in $G$, then $H + e$ forms a linear $\alpha$-semiradial in which every vertex is linear. 
This contradicts the maximality of $H$.  
As such, $u$ is absolute in $G$. 
This further implies from Lemma~\ref{lem:neigh2ear} and the statement \ref{item:noear} 
that there is a scoop diear relative to $H$ whose grip is $e$. 
This completes the proof. 
\end{proof}

\section{Grounds for Radials}  \label{sec:ground4r} 
\subsection{Strong and Almost Strong Grounds} \label{sec:ground4r:ground} 

In Section~\ref{sec:ground4r:ground}, we introduce the concept of strong and almost strong grounds for radials 
and provide some properties to be used in later sections. 
Lemma~\ref{lem:r2union} is easily confirmed. 

\begin{lemma} \label{lem:r2union} 
Let $\alpha\in\{+, -\}$. 
Let $G$ be a bidirected graph, and let $r\in V(G)$. 
Let $H_1$ and $H_2$ be subgraphs of $G$ that are $\alpha$-radials with root $r$.  
Then, $H_1 + H_2$ is also an $\alpha$-radial with root $r$. 
Furthermore, 
if $H_1$ and $H_2$ are both strong or almost strong, then $H_1 + H_2$ is also strong or almost strong. 
\end{lemma}

Under Lemma~\ref{lem:r2union}, strong and almost strong grounds of radials can be defined.

\begin{definition} 
Let $\alpha\in\{+, -\}$.  Let $G$ be an $\alpha$-radial with root $r$. 
Let $H$ be the maximum subgraph $H$ of $G$ that is a strong or almost strong $\alpha$-radial with root $r$. 
We call $H$ the {\em strong} or {\em almost strong ground} 
if it is a strong or almost strong $\alpha$-radial with root $r$, respectively. 
\end{definition}

If the root $r$ is strong, then $G$ has a strong ground but never has an almost strong ground. 
If $r$ is sublinear, then $G$ never has a strong ground but has an almost strong ground. 
A strong or almost strong ground in a radial is said to be {\em trivial} if it comprises a single vertex, that is, the root, and no edge.   
Strong grounds can never be trivial, whereas almost strong grounds can be trivial. 
We use these facts in the remainder of this paper sometimes without explicitly mentioning it.

\begin{lemma} \label{lem:abgneigh2sublinear} \label{lem:astgneigh2sublinear} 
Let $\alpha \in \{+, -\}$.  
Let $G$ be an $\alpha$-radial with root $r$.  
\begin{rmenum} 
\item \label{item:str} Assume that $r$ is strong in $G$, and let $H$ be the strong ground of $G$. 
Then, every vertex from $\parNei{G}{H}$ is sublinear in $G$. 
\item \label{item:sublinr} Assume that $r$ is sublinear in $G$, and let $H$ be the almost  strong ground of $G$.  
Then, every neighbor of $V(H)\setminus \{r\}$ is sublinear in $G$. 
\end{rmenum} 
\end{lemma} 
\begin{proof} 
We first prove \ref{item:sublinr}. 
Let $v\in \parNei{G}{ V(H)\setminus \{r\} }$. If $v = r$, then $v$ is sublinear by assumption. 
Let $v \neq r$. 
Let  $w\in V(H)\setminus \{r\}$ be a vertex that is adjacent to $v$ with an edge $e$, 
 and let $\beta\in\{+, -\}$ be the sign of $v$ over $e$. 
Suppose that $v$ is strong in $G$. 
Then, $G$ has a $(-\beta, -\alpha)$-ditrail $P$ from $v$ to $r$. 
Trace $P$ from $v$, and let $x$ be the first encountered vertex that is in $V(H)$. 

First, consider the case $x = r$. 
Because $r$ is sublinear, $vPx$ is a $(-\beta, -\alpha)$-ditrail. 
Let $\gamma$ be the sign of $w$ over $e$, and let $Q$ be a $(-\gamma, -\alpha)$-ditrail of $H$ from $w$ to $r$. 
Then, $xP^{-1}v + (v, e, w) + Q$ is a closed $(-\alpha, -\alpha)$-ditrail over $r$, 
which is a contradiction. 

Next, consider the case $x \in V(H)\setminus \{r\}$. 
Then, $(w, e, v) + vPx$ is a diear relative to $H$ that does not contain $r$. 
By Theorem~\ref{thm:asr2char}, this contradicts the maximality of $H$. 
The statement \ref{item:str} can be proved in a similar way using Theorem~\ref{thm:str}. 
The lemma is proved. 
\end{proof}

\subsection{Extended Grounds for Radials with Sublinear Root} 

In this section, we introduce the concept of extended grounds for radials with sublinear root and provide their properties.  
Lemma~\ref{lem:shell2sum} can easily be confirmed. 

\begin{lemma} \label{lem:shell2sum} 
Let $\alpha \in \{+, -\}$.  
Let $G$ be an $\alpha$-radial with sublinear root $r$.  
Let $H$ be the almost strong ground of $G$.  
Let $I_1$ and $I_2$  be subgraphs of $G$ 
such that, for each $i\in \{1,2\}$, 
 $I_i$ is an $\alpha$-radial with root $r$, $V(H)\subseteq V(I_i)$ holds,  and every vertex in $V(I_i)\setminus V(H)$ is sublinear in $G$. 
Then, $I_1 + I_2$ is also an  $\alpha$-radial with root $r$ that contains $V(H)$ 
for which every vertex in $V(I_1 + I_2)\setminus V(H)$ is sublinear in $G$. 
\end{lemma}

From Lemma~\ref{lem:shell2sum}, the concept of extended grounds can be introduced.

\begin{definition} 
Let $\alpha \in \{+, -\}$.  
Let $G$ be an $\alpha$-radial with sublinear root $r$.  
Let $H$ be the almost strong ground of $G$.  
The {\em extended ground} of $G$ is the maximum subgraph $I$ of $G$  
such that $I$ is an $\alpha$-radial with root $r$ and contains $H$,  and every vertex in $V(I)\setminus V(H)$ is sublinear in $G$. 
We call $V(I)\setminus V(H)$ the {\em shell} of $I$. 
Let $S_1$ be the subset of shell that is defined as follows: 
A vertex $x$ from a shell is in $S_1$ 
if $I$ has an $(\alpha, -\alpha)$-ditrail from $x$ to $r$ that is disjoint from $V(H)\setminus \{r\}$; 
 let $S_2 := V(I)\setminus V(H) \setminus S_1$. 
We call $S_1$ and $S_2$ the {\em first} and {\em second} shells, respectively. 
We call $G$ a {\em triplex} if $G$ is equal to its extended ground. 
\end{definition}

Lemmas~\ref{lem:rrneigh2g} and \ref{lem:shellneigh2strong} are provided to be used in later sections.

\begin{lemma} \label{lem:rrneigh2g} 
Let $\alpha\in\{+, -\}$.  Let $G$ be an $\alpha$-radial with sublinear root $r$. 
Let $v  \in \oNei{G}{r}{-\alpha}$.  
If $v$ is strong, then it is contained in the almost strong ground. 
If $v$ is sublinear, then it is contained in the first shell. 
\end{lemma} 
\begin{proof} 
Let $H$ and $I$ be the almost strong and extended ground, respectively. 
If $v$ is a strong vertex, 
then Lemma~\ref{lem:neigh2ear} 
implies that there is a simple $(-\alpha, -\alpha)$-diear or $-\alpha$-scoop diear with grip $vr$ that is relative to $r$;    
clearly, there is a $-\alpha$-scoop diear with grip $vr$. 
Hence, Theorem~\ref{thm:asr2char} implies $v \in V(H)$.  
If $v$ is a sublinear vertex, then the sign of $v$  over $vr$ is $\alpha$. 
Therefore, 
$G. vr$ is a subgraph of $I$ that is a sublinear $\alpha$-radial with root $r$. 
Consequently, it is contained in the first shell. 
This completes the proof of this lemma. 
\end{proof}

Lemma~\ref{lem:shellneigh2strong} can be easily confirmed from Lemma~\ref{lem:neigh2ear} 
by a similar discussion as in the proof of Lemma~\ref{lem:lg2neigh}.

\begin{lemma}  \label{lem:shellneigh2strong} 
Let $\alpha \in \{+, -\}$.  
Let $G$ be an $\alpha$-radial with sublinear root $r$.  
Let $I$ be the extended ground of $G$. 
Then, the following properties hold:  
\begin{rmenum} 
\item \label{item:shellneigh2strong:r2sublin} There is no simple $(-\alpha, -\alpha)$-diear relative to $I$. 
\item \label{item:shellneigh2strong:round} $\oNei{G}{I}{-\alpha} \neq \emptyset$, and every vertex from $\oNei{G}{I}{-\alpha}$ is strong in $G$. 
\end{rmenum} 
\end{lemma}

Observation~\ref{obs:shell} can easily be confirmed from the definition of shells. 

\begin{observation}  \label{obs:shell} 
Let $\alpha \in \{+, -\}$.  
Let $G$ be an $\alpha$-radial with sublinear root $r$.  
Let $H$ and $I$ be the almost strong  and  extended grounds of $G$, respectively.  
Let $S_1$ and $S_2$ be the first and second shells of $I$, respectively. 
Then, the following properties hold: 
\begin{rmenum} 
\item If $S_1 = S_2 = \emptyset$, then $G = I = H$. 
\item If $V(H) = \{r\}$, then $S_2 = \emptyset$. 
\end{rmenum} 
\end{observation}

In the following two lemmas, we investigate the inside structure of extended grounds 
and show that an extended ground is a combination of two radials and one semiradial. 
Lemma~\ref{lem:shell2path} is provided for proving Lemma~\ref{lem:shell}.

\begin{lemma} \label{lem:shell2path} 
Let $\alpha \in \{+, -\}$.  
Let $G$ be an $\alpha$-radial with sublinear root $r$.  
Let $H$ and $I$ be the almost strong  and  extended grounds of $G$, respectively.  
Let $S\subseteq V(G)$ be the shell of $I$, and let $S_1$ and $S_2$ be the first and second shells, respectively.  
Then, the following properties hold: 
\begin{rmenum} 
\item \label{item:shell2path:r} If $P$ is an $(\alpha, -\alpha)$-ditrail from $x\in S_1$ to $r$  with $V(P)\setminus \{r\} \subseteq S$, 
then $V(P)\setminus \{r\} \subseteq S_1$ holds. 
\item \label{item:shell2path:sr}  If $P$ is an $(\alpha, -\alpha)$-ditrail from $x\in S_2$ to $r$, 
then there exists a vertex $y\in V(H)\setminus \{r\}$ such that $V(xPy)\setminus \{y\} \subseteq S_2$. 
There is no $-\alpha$-ditrail from  any $x \in S_2$ to any $y\in V(H)\setminus \{r\}$  
whose vertices except $y$ are contained in $S_2$. 
\end{rmenum} 
\end{lemma} 
\begin{proof} 
First, we prove \ref{item:shell2path:r}. 
From the definition of shell,  
 every vertex from $S_1$ is sublinear. 
Therefore, for every $y\in V(P)$,  $yPr$ is an $(\alpha, -\alpha)$-ditrail with $V(yPr)\setminus \{r\} \subseteq S$. That is, $y\in S_1$. 
This implies $V(P) \setminus \{r\} \subseteq S_1$.  
Thus, \ref{item:shell2path:r} is proved. 

We next prove \ref{item:shell2path:sr}. 
Trace $P$ from $x$, and let $y$ be the first encountered vertex that is in $S_1 \cup  V(H)$.  
If $y$ is $r$, then $xPy$ is an $(\alpha, -\alpha)$-ditrail from $x$ to $r$ with $V(xPy)\setminus \{r\} \subseteq S$, 
which contradicts $x\in S_2$. 
Hence, $y \neq r$. 
Suppose $y\in S_1$. 
Because $y$ is sublinear in $G$, $xPy$ is an $(\alpha, -\alpha)$-ditrail. 
By \ref{item:shell2path:r}, there is an $(\alpha, -\alpha)$-ditrail $Q$ from $y$ to $r$ with $V(zQr)\setminus \{r\} \subseteq S_1$.  
Therefore, $xPy + Q$ is an $(\alpha, -\alpha)$-ditrail from $x$ to $r$ with $V(P+Q)\setminus \{r\} \subseteq S$.  
This contradicts $x\in S_2$. 
Therefore, $y\in V(H)\setminus \{r\}$ follows, and $V(xPy)\setminus \{y\} \subseteq S_2$ is proved. 
The remaining claim of \ref{item:shell2path:sr} can easily be proved by considering the concatenation of ditrails. 
The statement \ref{item:shell2path:sr} is proved. 
This completes the proof of the lemma. 
\end{proof}

Lemma~\ref{lem:shell} is derived from Lemma~\ref{lem:shell2path} and is used in later sections.

\begin{lemma} \label{lem:shell} 
Let $\alpha \in \{+, -\}$.  
Let $G$ be an $\alpha$-radial with sublinear root $r$.  
Let $H$ and $I$ be the almost strong  and  extended grounds of $G$, respectively.  
Let $S_1$ and $S_2$ be the first and second shells of $I$, respectively.  
Then, the following properties hold: 
\begin{rmenum} 
\item \label{item:shell:r} $G[S_1 \cup \{r\} ]$ is a sublinear $\alpha$-radial with root $r$. 
\item \label{item:shell:sr} $G[S_2 \cup V(H) ]/V(H) - E_G[r, S_2]$ is a linear $\alpha$-semiradial with root $s$, 
where $s$ denote the contracted vertex that corresponds to $V(H)$. 
\item \label{item:shell:lgcut}  For every edge of the form $uv$ such that $u\in S_1$ and $v\in  V(H)\cup S_2$, the sign of $u$ over $uv$ is $\alpha$. 
\item \label{item:shell:rcut} For every edge of the form $rv$ such that $v\in  S_2$, the sign of $r$ over $rv$ is $\alpha$. 
\end{rmenum} 
\end{lemma} 
\begin{proof} 
We first prove \ref{item:shell:r}. 
Lemma~\ref{lem:shell2path} \ref{item:shell2path:r} easily implies that $G[S_1 \cup \{r\}]$ is an $\alpha$-radial with root $r$. 
From the definition of first shell, every vertex from $S_1$ is sublinear in $G$. 
Hence, the radial $G[S_1 \cup \{r\}]$ is sublinear. 
The statement \ref{item:shell:r} is proved.  
Lemmas~\ref{lem:shell2path} \ref{item:shell2path:sr} easily implies  \ref{item:shell:sr}.

We next prove \ref{item:shell:lgcut}. 
Suppose, to the contrary, that the sign of $u$ over $uv$ is $-\alpha$. 
Let $\beta \in \{+, -\}$ be the sign of $v$ over $uv$. 
By Lemma~\ref{lem:shell2path} \ref{item:shell2path:r}, 
there is an $(\alpha, -\alpha)$-ditrail from $Q$ from $u$ to $r$ with $V(Q)\setminus \{r\} \subseteq S_1$. 
Then, $(v, uv, u) + Q$ is a $(\beta, -\alpha)$-ditrail from $v$ to $r$. 
If $v\in S_2$ holds,  it contradicts either $v\not\in S_1$ or that $v$ is sublinear in $G$. 
For the case $v\in V(H)\setminus \{r\}$, let $R$ be a $(-\beta, -\alpha)$-ditrail in $H$ from $v$ to $r$. 
Then, $Q^{-1} + (u, uv, v) + R$ is a closed $(-\alpha, -\alpha)$-ditrail over $r$. 
This contradicts the assumption that $r$ is sublinear. 
Thus, \ref{item:shell:lgcut} follows. 
The statement \ref{item:shell:rcut} can be proved in a similar way by considering the concatenation of ditrails. 
The lemma is proved. 
\end{proof}

\section{Round Radials} 

\begin{definition} 
Let $\alpha\in\{+, -\}$. 
An $\alpha$-radial with sublinear root $r$ is said to be {\em round} 
if every vertex from $\oNei{G}{r}{-\alpha}$ is strong. 
\end{definition}

Lemma~\ref{lem:asg2neigh} is easily implied from  Lemma~\ref{lem:rrneigh2g}. 

\begin{lemma} \label{lem:asg2neigh} 
Let $\alpha\in\{+, -\}$, and let $G$ be a round $\alpha$-radial with root $r\in V(G)$. 
Then, every vertex from $\oNei{G}{r}{-\alpha}$ is contained in the almost strong ground. 
\end{lemma} 

By Lemma~\ref{lem:asg2neigh}, 
an $\alpha$-radial with sublinear root $r$ is round 
if and only if every vertex from $\oNei{G}{r}{-\alpha}$ is contained in the almost strong ground.

\section{Contraction Lemmas} 

In this section, we provide Lemmas~\ref{lem:const4l} and \ref{lem:const4s} to be used in proving lemmas in Section~\ref{sec:decomposition}. 

\begin{lemma} \label{lem:const4l} 
Let $G$ be a bidirected graph, 
and let $H$ be an induced subgraph of $G$ with $r\in V(G)\cap V(H)$. 
Let $x\in V(G)\setminus V(H)$ and $\beta\in\{+, -\}$. 
\begin{rmenum} 
\item \label{item:const4l:r} 
Assume $\parNei{G}{V(G)\setminus V(H)} \subseteq \reachset{H}{r}{\alpha}{-\alpha}\setminus  \reachset{G}{r}{-\alpha}{-\alpha}$. 
Then, $G$ has a $(\beta, -\alpha)$-ditrail from $x$ to $r$ if and only if 
$G$ has a $(\beta, -\alpha)$-ditrail from $x$ to a vertex in $V(H)$ whose edges are disjoint from $E(H)$. 
\item \label{item:const4l:sr} 
Assume $\parNei{G}{V(G)\setminus V(H)} \subseteq \reachsets{H}{r}{\alpha} \setminus  \reachsets{G}{r}{-\alpha}$. 
Then, 
$G$ has a $\beta$-ditrail from $x$ to $r$ if and only if 
$G$ has a $(\beta, -\alpha)$-ditrail from $x$ to a vertex in $V(H)$ whose edges are disjoint from $E(H)$. 
\end{rmenum} 
\end{lemma} 
\begin{proof} 
First, we prove \ref{item:const4l:r}. 
For proving the sufficiency, let $P$ be a $(\beta, -\alpha)$-ditrail in $G$ from $x$ to $r$. 
Trace $P$ from $x$, and let $y$ be the first encountered vertex that is in $V(H)$.  
If $xPy$ is a $(\beta, \alpha)$-ditrail, then $yPr$ is a $(-\alpha, -\alpha)$-ditrail, 
which contradicts $y\not\in \reachset{G}{r}{-\alpha}{-\alpha}$.   
Hence, $xPy$ is a $(\beta, -\alpha)$-ditrail that satisfies the condition. 

The necessity can easily be proved by considering the concatenation of ditrails. 
This proves \ref{item:const4l:r}. 
The statement \ref{item:const4l:sr} can be proved by a similar discussion.  
\end{proof}

\begin{lemma} \label{lem:const4s} 
Let $G$ be a bidirected graph, 
and let $H$ be an induced subgraph of $G$ with $r\in V(G)\cap V(H)$. 
Let $x\in V(G)\setminus V(H)$ and $\beta\in\{+, -\}$. 
\begin{rmenum} 
\item \label{item:const4s:r} 
Assume $\parNei{G}{V(G)\setminus V(H)} \subseteq \reachset{H}{r}{\alpha}{-\alpha} \cap \reachset{H}{r}{-\alpha}{-\alpha}$. 
Then, $G$ has an $(\beta, -\alpha)$-ditrail from $x$ to $r$ if and only if 
$G$ has a $\beta$-ditrail from $x$ to a vertex in $V(H)$ whose edges are disjoint from $E(H)$. 
\item \label{item:const4s:sr} 
Assume $\parNei{G}{V(G)\setminus V(H)} \subseteq \reachsets{H}{r}{\alpha} \cap \reachsets{H}{r}{-\alpha}$. 
Then, 
$G$ has a $\beta$-ditrail from $x$ to $r$ if and only if 
$G$ has a $\beta$-ditrail from $x$ to a vertex in $V(H)$ whose edges are disjoint from $E(H)$. 
\end{rmenum} 
\end{lemma} 
\begin{proof} 
We first prove the sufficiency of \ref{item:const4s:r}. 
Let $P$ be a $(\beta, -\alpha)$-ditrail from $x$ to $r$. 
Trace $P$ from $x$, and let $y$ be the first encountered vertex that is in $V(H)$. 
Then, $xPy$ is a desired ditrail. 
The necessity  can be proved by considering the concatenation of ditrails. 
This proves \ref{item:const4s:r}. 
The statement \ref{item:const4s:sr} can be proved in a similar way. 
\end{proof}

\section{Decomposition of Radials and Semiradials}  \label{sec:decomposition} 
\subsection{From General Semiradials to Sharp Semiradials}

Lemma~\ref{lem:const4s} \ref{item:const4s:sr} easily implies Lemma~\ref{lem:srpath2abg}.

\begin{lemma} \label{lem:srpath2abg} 
Let $\alpha\in\{+, -\}$. 
Let $G$ be an $\alpha$-semiradial with root $r\in V(G)$. 
Let $H$ be the absolute ground of $G$. 
Let $x\in V(G)$ and $\beta\in\{+, -\}$. 
Then, there is a $\beta$-ditrail from $x$ to $r$ 
if and only if 
there is a $\beta$-ditrail from $x$ to a vertex in $V(H)$ whose edges are disjoint from $E(H)$. 
\end{lemma}

Lemmas~\ref{lem:abgnei2lin} and \ref{lem:srpath2abg} easily imply Lemma~\ref{lem:abg2lg}.

\begin{lemma} \label{lem:abg2lg}
Let $\alpha\in\{+, -\}$. 
Let $G$ be an $\alpha$-semiradial with root $r\in V(G)$. 
Let $H$ be the absolute ground of $G$. 
Then, $G/H$ is a sharp $\alpha$-semiradial with root $h$, where $h$ is the contracted vertex that  corresponds to $H$. 
\end{lemma} 
\begin{proof} 
Lemma~\ref{lem:srpath2abg} easily implies that $G/H$ is an $\alpha$-semiradial with root $h$. 
Lemma~\ref{lem:abgnei2lin} further implies that it is sharp. 
The lemma is proved. 
\end{proof}

\subsection{From General Radials with Strong Root to Sharp Semiradials}

Lemma~\ref{lem:const4s} \ref{item:const4s:r} easily implies Lemma~\ref{lem:rpath2stg}.

\begin{lemma} \label{lem:rpath2stg} 
Let $\alpha \in \{+, -\}$.  
Let $G$ be an $\alpha$-radial with strong root $r$.  
Let $H$ be the strong ground of $G$. 
Let $x\in V(G)$ and $\beta \in \{+, -\}$. 
Then, $G$ has a $(\beta, -\alpha)$-ditrail from $x$ to $r$ 
if and only if $G$ has a $\beta$-ditrail from $x$ to a vertex in $V(H)$ whose edges are disjoint from $E(H)$. 
\end{lemma}

Lemmas~\ref{lem:abgneigh2sublinear} and \ref{lem:rpath2stg} imply Lemma~\ref{lem:stg2lg}.

\begin{lemma} \label{lem:stg2lg} 
Let $\alpha \in \{+, -\}$.  
Let $G$ be an $\alpha$-radial with strong root $r$.  
Let $H$ be the strong ground of $G$. 
Then, $G/H$ is a sharp $\alpha$-semiradial with root $h$, 
where $h$ is the contracted vertex that corresponds to $H$. 
\end{lemma} 
\begin{proof} 
Lemma~\ref{lem:rpath2stg} implies that $G/H$ is an $\alpha$-semiradial with root $h$. 
Lemma~\ref{lem:abgneigh2sublinear} implies that every neighbor of $h$ is linear in $G/H$. 
Hence, $G/H$ is sharp. 
\end{proof}

\subsection{From General Radials with  Sublinear Root to Sharp Semiradials}

Lemma~\ref{lem:rrdel} can easily be implied from Lemma~\ref{lem:r2delete}.  
We use Lemma~\ref{lem:rrdel} to prove Lemmas~\ref{lem:path2shell} and \ref{lem:rpath2asg}.

\begin{lemma} \label{lem:rrdel} 
Let $\alpha\in\{+, -\}$, and let $G$ be an $\alpha$-radial with sublinear root $r\in V(G)$. 
Let $F \subseteq \ocut{G}{r}{\alpha}$. 
Then, $\reachset{G}{r}{\beta}{-\alpha} = \reachset{G - F}{r}{\beta}{-\alpha}$ for each $\beta\in\{+, -\}$. 
\end{lemma} 
\begin{proof} 
Because $\{r\} \subseteq \reachset{G}{r}{-\alpha}{-\alpha}$ holds, 
Lemma~\ref{lem:r2delete} implies the claim. 
\end{proof}

Lemmas~\ref{lem:rrneigh2g}, \ref{lem:const4s} \ref{item:const4s:r}, and \ref{lem:rrdel} imply Lemma~\ref{lem:path2shell}.

\begin{lemma} \label{lem:path2shell} 
Let $\alpha \in \{+, -\}$.  
Let $G$ be an $\alpha$-radial with sublinear root $r$.  
Let $I$ be the extended ground of $G$.  
Let $x\in V(G)\setminus V(I) $, and let $\beta\in \{+, -\}$. 
Then, the following statements are equivalent: 
\begin{rmenum} 
\item \label{item:path2shell:r} There is a $(\beta, -\alpha)$-ditrail from $x$ to $r$ in $G$; 
\item \label{item:path2shell:del} there is a $(\beta, -\alpha)$-ditrail from $x$ to $r$ in $G - E_G[r, V(G)\setminus V(I)]$; and, 
\item \label{item:path2shell:cut} there is a $(\beta, -\alpha)$-ditrail from $x$ to a vertex in $V(H)\setminus \{r\}$ whose edges are disjoint from $E(I) \cup E_G[r, V(G)\setminus V(I)]$. 
\end{rmenum} 
\end{lemma} 
\begin{proof} 
Lemma~\ref{lem:rrneigh2g} implies $E_G[r, V(G)\setminus V(I)] \subseteq \ocut{G}{r}{\alpha}$. 
Hence, by Lemma~\ref{lem:rrdel}, 
 the statements~\ref{item:path2shell:r} and \ref{item:path2shell:del} are equivalent. 
It is easily confirmed from Lemma~\ref{lem:const4s} \ref{item:const4s:r} 
that \ref{item:path2shell:del} and \ref{item:path2shell:cut} are equivalent. 
The lemma is proved. 
\end{proof}

Lemmas~\ref{lem:shellneigh2strong} and \ref{lem:path2shell} imply Lemma~\ref{lem:shell2astg}.

\begin{lemma} \label{lem:shell2astg} 
Let $\alpha \in \{+, -\}$.  
Let $G$ be an $\alpha$-radial with sublinear root $r$.  
Let $I$ be the extended ground of $G$.  
Then, $G/I$ is a round $\alpha$-radial with sublinear root $i$, where $i$ is the contracted vertex that corresponds to $I$. 
\end{lemma} 
\begin{proof} 
Lemmas~\ref{lem:shellneigh2strong} \ref{item:shellneigh2strong:r2sublin} and  \ref{lem:path2shell} imply that $G/I$ is an $\alpha$-radial with sublinear root $i$. 
Lemma~\ref{lem:shellneigh2strong} \ref{item:shellneigh2strong:round} further implies that every vertex from $\ocut{G/I}{i}{-\alpha}$ is strong. 
Hence, $G/I$ is round.  
The lemma is proved. 
\end{proof}

\subsection{From Sharp Semiradials to Round Radials}

Lemma~\ref{lem:srpath2lg} can easily be confirmed from Lemma~\ref{lem:const4l} \ref{item:const4l:sr}.

\begin{lemma} \label{lem:srpath2lg}
Let $\alpha\in\{+, -\}$, and let $G$ be a sharp $\alpha$-semiradial with root $r\in V(G)$. 
Let $H$ be the linear ground of $G$.  
Let $x\in V(G)\setminus V(H)$ and $\beta \in \{+, -\}$. 
Then,  $G$ has a $\beta$-ditrail from $x$ to $r$ 
if and only if 
$G$ has a $(\beta, -\alpha)$-ditrail from $x$ to a vertex in $V(H)$ whose edges are disjoint from $E(H)$. 
\end{lemma}

Lemmas~\ref{lem:lg2neigh} and \ref{lem:srpath2lg} imply Lemma~\ref{lem:lg2asg}.

\begin{lemma} \label{lem:lg2asg}
Let $\alpha\in\{+, -\}$, and let $G$ be a sharp $\alpha$-semiradial with root $r\in V(G)$. 
Let $H$ be the linear ground of $G$. 
Then,  $G/H$ is a round $\alpha$-radial with root $h$, where $h$ is the contracted vertex that corresponds to $H$. 
\end{lemma} 
\begin{proof} 
It  easily follows from Lemma~\ref{lem:srpath2lg} that $G/H$ is an $\alpha$-radial with root $h$. 
Under Lemma~\ref{lem:srpath2lg}, Lemma~\ref{lem:lg2neigh} \ref{item:noear} further implies that $h$ is sublinear in $G/ H$. 
Additionally, Lemma~\ref{lem:lg2neigh} \ref{item:scoop} implies that every vertex from $\oNei{G/H}{h}{-\alpha}$ is strong in $G/H$. 
Thus, it follows that $G/H$ is round. 
\end{proof}

\subsection{From Round Radials to Sharp Semiradials}

Lemmas~\ref{lem:asg2neigh}, \ref{lem:const4s}, and \ref{lem:rrdel} imply Lemma~\ref{lem:rpath2asg}.

\begin{lemma} \label{lem:rpath2asg} 
Let $\alpha\in\{+, -\}$, and let $G$ be a round $\alpha$-radial with root $r\in V(G)$. 
Let $H$ be the almost strong ground of $G$. 
Let $x\in V(G)\setminus V(H)$ and $\beta \in \{+, -\}$. 
Then, the following statements are equivalent: 
\begin{rmenum}
\item \label{item:rpath2asg:original} 
$G$ has a $(\beta, -\alpha)$-ditrail from $x$ to $r$. 
\item \label{item:rpath2asg:cut}
$G - E_G[r, V(G)\setminus V(H)]$ has a $(\beta, -\alpha)$-ditrail from $x$ to $r$. 
\item \label{item:rpath2asg:ground}
$G - E_G[r, V(G)\setminus V(H)]$ has a $\beta$-ditrail from $x$ to a vertex in $V(H)\setminus \{r\}$ whose edges are disjoint from $E(H)$. 
\end{rmenum}  
\end{lemma} 
\begin{proof} 
From Lemma~\ref{lem:asg2neigh},  we have $E_G[r, V(G)\setminus V(H)] \subseteq \ocut{G}{r}{\alpha}$. 
Hence, Lemma~\ref{lem:rrdel} implies that  \ref{item:rpath2asg:original} and \ref{item:rpath2asg:cut} are equivalent. 
Lemma~\ref{lem:const4s} \ref{item:const4s:r} implies  that  \ref{item:rpath2asg:original} and \ref{item:rpath2asg:ground} are equivalent.  
Hence, the lemma is proved. 
\end{proof}

Lemmas~\ref{lem:astgneigh2sublinear} and \ref{lem:rpath2asg} imply Lemma~\ref{lem:asg2lg}.

\begin{lemma} \label{lem:asg2lg}
Let $\alpha\in\{+, -\}$, and let $G$ be a round $\alpha$-radial with root $r\in V(G)$. 
Let $H$ be the almost strong ground of $G$. 
Let $G' :=  G - E_G[r, V(G)\setminus V(H)] /H$.    
Then, $G'$ is a sharp $\alpha$-semiradial with root $h$, where $h$ is the contracted vertex that corresponds to $H$. 
\end{lemma} 
\begin{proof} 
Lemma~\ref{lem:rpath2asg} easily implies that $G/H$ is an $\alpha$-semiradial with root $h$.  
Under Lemma~\ref{lem:rpath2asg}, 
 Lemma~\ref{lem:astgneigh2sublinear} further implies that every neighbor of $h$ is linear in $G/H$. 
 Thus,  $G/H$ is sharp. 
\end{proof}

\section{Gluing Lemmas} \label{sec:constlem} 
\subsection{Gluing Sharp Semiradials onto Sets of Absolute and Strong Vertices} \label{sec:constlem:abst} 

In Section~\ref{sec:constlem:abst}, 
we provide Lemmas~\ref{lem:construct2r} and \ref{lem:construct2sr} to be used in Section~\ref{sec:construct}.

\begin{lemma} \label{lem:construct2r} 
Let $\alpha \in \{+, -\}$. 
Let $G$ and $H$ be disjoint bidirected graphs 
such that $G$ is a sharp $\alpha$-semiradial with root $s$, and  $H$ is an $\alpha$-radial with root $r$. 
Let $S\subseteq V(H)$ be a set of strong vertices in $H$. 
Let $\hat{G} := (G; s) \oplus (H; S)$. 
Then,  for each $\beta\in \{+, -\}$, 
\begin{rmenum}
\item  $\reachsets{G}{s}{\beta} \setminus \{s\}$ is equal to $\reachset{\hat{G}}{r}{\beta}{-\alpha} \cap ( V(G)\setminus \{s\} )$;  and, 
\item $\reachset{H}{r}{\beta}{-\alpha}$ is equal to $\reachset{\hat{G}}{r}{\beta}{-\alpha} \cap V(H)$. 
\end{rmenum} 
\end{lemma} 
\begin{proof} 
It can easily be confirmed that $\reachsets{G}{s}{\beta} \setminus \{s\}$ is contained in $\reachset{\hat{G}}{r}{\beta}{-\alpha} \cap ( V(G)\setminus \{s\} )$ 
by considering the concatenation of ditrails. 
It is also clear that 
$\reachset{H}{r}{\beta}{-\alpha}$ is a subset of $\reachset{\hat{G}}{r}{\beta}{-\alpha} \cap V(H)$. 
Hence, in the following, 
we prove 
$\reachsets{G}{s}{\beta}\setminus \{s\} \supseteq \reachset{\hat{G}}{r}{\beta}{-\alpha} \cap (V(G)\setminus \{s\})$ 
and $\reachset{H}{r}{\beta}{-\alpha} \supseteq \reachset{\hat{G}}{r}{\beta}{-\alpha} \cap V(H)$. 
Let $x \in \reachset{\hat{G}}{r}{\beta}{-\alpha}$, 
and let $P$ be a $(\beta, -\alpha)$-ditrail in $\hat{G}$ from $x$ to $r$.  

First, consider the case  $x\in V(G)\setminus \{s\}$. 
Trace $P$ from $x$, and let $y$ be the first encountered vertex that is in $V(H)$. 
Then, $xPy$ corresponds to a $\beta$-ditrail from $x$ to $s$ in $G$. 
This proves $x\in \reachsets{G}{s}{\beta}$.  

Next, consider the case $x\in V(H)$. 
Because $G$ is sharp, Theorem~\ref{thm:asr} implies that $P$ cannot contain a simple diear relative to $H$ as its subditrail. 
This implies that $P$ is a ditrail of $H$ for this case. 
Therefore, $x\in \reachset{H}{r}{\beta}{-\alpha}$ holds. 
This completes the proof. 
\end{proof}

Lemma~\ref{lem:construct2sr} is  an analogue of Lemma~\ref{lem:construct2r} 
and can easily be confirmed by a similar discussion.

\begin{lemma} \label{lem:construct2sr} 
Let $\alpha \in \{+, -\}$. 
Let $G$ and $H$ be disjoint bidirected graphs 
such that $G$ is a sharp $\alpha$-semiradial with root $s$, and  $H$ is an $\alpha$-semiradial with root $r$. 
Let $S\subseteq V(H)$ be a set of absolute vertices in $H$. 
Let $\hat{G} := (G; s) \oplus (H; S)$. 
Then,  for each $\beta\in \{+, -\}$, 
\begin{rmenum}
\item  $\reachsets{G}{s}{\beta} \setminus \{s\}$ is equal to $\reachsets{\hat{G}}{r}{\beta} \cap ( V(G)\setminus \{s\} )$;  and, 
\item  $\reachsets{H}{r}{\beta}$ is equal to $\reachsets{\hat{G}}{r}{\beta}\cap V(H)$. 
\end{rmenum} 
\end{lemma}

\subsection{Gluing Round Radials onto Sets of Linear and Sublinear Vertices}  \label{sec:constlem:linsublin} 

In Section~\ref{sec:constlem:linsublin}, 
we introduce more lemmas to be used in Section~\ref{sec:construct}.  
First, we provide Lemma~\ref{lem:trimmedr2const} to be used for proving Lemma~\ref{lem:construct2lin}.

\begin{lemma}  \label{lem:trimmedr2const} 
Let $\alpha \in \{+, -\}$.  
Let $G$ and $H$ be disjoint bidirected graphs 
such that $G$ is an $\alpha$-radial with sublinear root $s$ for which $\ocut{G}{s}{\alpha} = \emptyset$, 
and $H$ is an $\alpha$-semiradial with root $r$. 
Let $S \subseteq V(H)\setminus \reachsets{H}{r}{-\alpha}$. 
Let $\hat{G} := (G; s) \oplus (H; S)$. 
Then,  for each $\beta\in \{+, -\}$, 
\begin{rmenum}
\item  $\reachset{G}{s}{\beta}{-\alpha} \setminus \{s\}$ is equal to $\reachsets{\hat{G}}{r}{\beta} \cap ( V(G)\setminus \{s\} )$; and, 
\item  $\reachsets{H}{r}{\beta}$ is equal to $\reachsets{\hat{G}}{r}{\beta} \cap V(H)$. 
\end{rmenum} 
\end{lemma} 
\begin{proof} 
By considering the concatenation of ditrails,  it can easily be confirmed 
that $\reachset{G}{s}{\beta}{-\alpha} \setminus \{s\}$ and $\reachsets{H}{r}{\beta}$ 
are subsets of $\reachsets{\hat{G}}{r}{\beta} \cap ( V(G)\setminus \{s\} )$ and $\reachsets{\hat{G}}{r}{\beta} \cap V(H)$, respectively. 
In the following, 
we prove that 
$\reachset{G}{s}{\beta}{-\alpha} \setminus \{s\}$ and $\reachsets{H}{r}{\beta}$ contain 
$\reachsets{\hat{G}}{r}{\beta} \cap ( V(G)\setminus \{s\} )$ and $\reachsets{\hat{G}}{r}{\beta} \cap V(H)$, respectively.  

Let $x\in V(\hat{G})$, and let $P$ be a $\beta$-ditrail in $\hat{G}$ from $x$ to $r$. 
The assumption on $G$ implies that $P$ cannot contain a simple diear relative to $H$ as a subditrail.  
This further implies that $P$ shares at most one edge from $\parcut{\hat{G}}{H}$. 
Therefore, if $x$ is a vertex from $V(G) \setminus \{s\}$, then 
there is a vertex $y\in V(P)$ such that 
$E(xPy) \subseteq E(G)$ and $E(yPr) \subseteq E(H)$ hold. 
Because $y \in S$ holds, $yPr$ is an $\alpha$-ditrail. 
Hence, $xPy$ is an $(\beta, -\alpha)$-ditrail. Thus, $x \in \reachset{G}{s}{\beta}{-\alpha} \setminus \{s\}$ is obtained. 
By contrast, if $x$ is a vertex from $V(H)$, then $P$ is a ditrail of $H$, 
and $x$ is accordingly a vertex from $\reachsets{H}{r}{\beta}$. 
Thus, the lemma is proved. 
\end{proof}

Lemma~\ref{lem:construct2lin} is derived from Lemmas~\ref{lem:edgeaddsr}, \ref{lem:sr2delete}, and \ref{lem:trimmedr2const}.

\begin{lemma} \label{lem:construct2lin} 
Let $\alpha \in \{+, -\}$.  
Let $G$ and $H$ be disjoint bidirected graphs 
such that $G$ is an $\alpha$-radial with sublinear root $s$, and $H$ is an $\alpha$-semiradial with root $r$. 
Let $S \subseteq V(H)\setminus \reachsets{H}{r}{-\alpha}$. 
Let $\hat{G} := (G; s) \oplus (H; S)$. 
Then,  for each $\beta\in \{+, -\}$, 
\begin{rmenum}
\item  $\reachset{G}{s}{\beta}{-\alpha} \setminus \{s\}$ is equal to $\reachsets{\hat{G}}{r}{\beta} \cap ( V(G)\setminus \{s\} )$; and, 
\item $\reachsets{H}{r}{\beta}$ is equal to $\reachsets{\hat{G}}{r}{\beta} \cap V(H)$. 
\end{rmenum} 
\end{lemma} 
\begin{proof} 
Let $G' := G - \ocut{G}{s}{\alpha}$. 
Let $\hat{G}' := (G'; s) \oplus (H; S)$. 
Note $\hat{G}' = \hat{G} - F$, where $F \subseteq E(\hat{G})$ is the set of edges that corresponds to $\ocut{G}{s}{\alpha}$. 
Lemma~\ref{lem:sr2delete} implies  
$\reachset{G}{r}{\beta}{-\alpha} = \reachset{G'}{r}{\beta}{-\alpha}$ for each $\beta\in\{+, -\}$. 
Therefore, Lemma~\ref{lem:trimmedr2const} can be applied to $G'$, by which we obtain that 
$\reachset{G'}{s}{\beta}{-\alpha} \setminus \{s\}$ is equal to $\reachsets{\hat{G}'}{r}{\beta} \cap ( V(G)\setminus \{s\} )$, and  
$\reachsets{H}{r}{\beta}$ is equal to $\reachsets{\hat{G}'}{r}{\beta} \cap V(H)$. 
This also implies that Lemma~\ref{lem:edgeaddsr} can be applied to $\hat{G}'$ 
for adding $F$ to $\hat{G}'$, 
by which we obtain 
$\reachsets{\hat{G}}{r}{\beta} = \reachsets{G'}{r}{\beta}$ for each $\beta\in\{+, -\}$. 
Thus, the lemma is proved. 
\end{proof}

Lemma~\ref{lem:construct2sublin} is an analogue of Lemma~\ref{lem:construct2lin} and can easily confirmed by a similar discussion.

\begin{lemma} \label{lem:construct2sublin} 
Let $\alpha \in \{+, -\}$.  
Let $G$ and $H$ be disjoint bidirected graphs 
such that $G$ is an $\alpha$-radial with sublinear root $s$, and $H$ is an $\alpha$-radial with root $r$. 
Let $S \subseteq  V(H)\setminus \reachset{H}{r}{-\alpha}{-\alpha}$.  
Let $\hat{G} := (G; s) \oplus (H; S)$. 
Then,  for each $\beta\in \{+, -\}$, 
\begin{rmenum}
\item $\reachset{G}{s}{\beta}{-\alpha} \setminus \{s\}$ is equal to $\reachset{\hat{G}}{r}{\beta}{-\alpha} \cap ( V(G)\setminus \{s\} )$; and, 
\item $\reachset{H}{r}{\beta}{-\alpha}$ is equal to $\reachset{\hat{G}}{r}{\beta}{-\alpha} \cap V(H)$. 
\end{rmenum} 
\end{lemma}

\section{Construction of Radials and Semiradials}   \label{sec:construct} 
\subsection{Construction of Round Radials and Sharp Semiradials}  \label{sec:construct:rrlsr}

In Section~\ref{sec:construct:rrlsr}, 
we provide Lemmas~\ref{lem:const2srlg} and \ref{lem:const2rasg} to be used in Section~\ref{sec:characterization}. 
These lemmas are considered the converses of Lemmas~\ref{lem:lg2asg} and \ref{lem:asg2lg}, respectively. 
Lemma~\ref{lem:const2srlg} is  easily implied from Lemma~\ref{lem:construct2lin}. 

\begin{lemma} \label{lem:const2srlg}
Let $\alpha \in \{+, -\}$.  
Let $G$ be a round $\alpha$-radial with root $s$. 
Let $H$ be a linear $\alpha$-semiradial with root $r$.  
Assume that $G$ and $H$ are disjoint.  
Let $\hat{G} := (G; s) \oplus (H; V(H)\setminus \{r\})$.  
Then, $\hat{G}$ is a sharp $\alpha$-semiradial with root $r$ whose linear ground is $H$. 
\end{lemma}

For proving Lemma~\ref{lem:const2rasg}, we first provide Lemma~\ref{lem:rradd}. 
Lemma~\ref{lem:rradd} can easily be implied from Lemma~\ref{lem:edgeaddr}, 
and is also used for proving Lemma~\ref{lem:const2rrtriplex}.

\begin{lemma} \label{lem:rradd} 
Let $\alpha \in \{+, -\}$.  
Let $G$ be an $\alpha$-radial with sublinear root $r$. 
Let $G'$ be a bidirected graph obtained by adding some edges between $r$ and other vertices 
in which the sign of $r$ is $\alpha$. 
Then, $\reachset{G}{r}{\beta}{-\alpha} = \reachset{G'}{r}{\beta}{-\alpha}$ holds for each $\beta\in \{+, -\}$. 
\end{lemma} 
\begin{proof} 
Because $G. \{r\}$ is a trivial $\alpha$-radial, and $\{r\}\cap \reachset{G}{r}{-\alpha}{-\alpha} = \emptyset$, 
Lemma~\ref{lem:edgeaddr} implies the claim. 
\end{proof}

Lemmas~\ref{lem:construct2r} and \ref{lem:rradd} imply Lemma~\ref{lem:const2rasg}.

\begin{lemma} \label{lem:const2rasg} 
Let $\alpha \in \{+, -\}$.  
Let $G$ be a sharp $\alpha$-semiradial with  root $s$. 
Let $H$ be an almost strong $\alpha$-radial with root $r$. 
Assume that $G$ and $H$ are disjoint.  
Let $\hat{G}$ be a gluing sum $(G; s) \oplus (H; V(H)\setminus \{r\})$ 
or a bidirected graph obtained from this graph by adding some edges joining $r$ and $E_G[r, V(G)\setminus \{s\}]$ 
in which $r$ has the sign $\alpha$.  
Then, $\hat{G}$ is a round $\alpha$-radial with root $r$ whose almost strong ground is $H$. 
\end{lemma} 
\begin{proof} 
Let $\hat{G}'$ be a gluing sum $(G; r) \oplus (H; V(H)\setminus \{s\})$. 
Then, Lemma~\ref{lem:construct2r} implies that $\hat{G}'$ is a round $\alpha$-radial with root $r$ 
in which $H$ is the almost strong ground. 
Lemma~\ref{lem:rradd} further implies the claim for $\hat{G}$. 
\end{proof}

\subsection{Construction of General Radials and Semiradials}  \label{sec:construct:rsr}

In Section~\ref{sec:construct:rsr}, 
we provide Lemmas~\ref{lem:const2srag}, \ref{lem:const2rstg}, \ref{lem:triplex2const}, and \ref{lem:const2rrtriplex} 
to be used in Section~\ref{sec:characterization}. 
These lemmas are the converses of Lemmas~\ref{lem:abg2lg}, \ref{lem:stg2lg}, \ref{lem:shell}, and \ref{lem:shell2astg}, respectively. 
Lemma~\ref{lem:const2srag} is easily implied from Lemma~\ref{lem:construct2sr}.

\begin{lemma} \label{lem:const2srag} 
Let $\alpha \in \{+, -\}$.   
Let $G$ be a sharp $\alpha$-semiradial with root $s$.   
Let $H$ be an absolute $\alpha$-semiradial with root $r$. 
Assume that $G$ and $H$ are disjoint.  
Let $\hat{G} := (G; s) \oplus (H; V(H))$.  
Then, $\hat{G}$ is an $\alpha$-semiradial with root $r$  whose absolute ground is $H$. 
\end{lemma}

Lemma~\ref{lem:const2rstg} is easily implied from Lemma~\ref{lem:construct2r}.

\begin{lemma} \label{lem:const2rstg} 
Let $\alpha \in \{+, -\}$.  
Let $G$ and $H$ be disjoint bidirected graphs such that 
$G$ is a sharp $\alpha$-semiradial with root $s$, 
and  $H$ is a strong $\alpha$-radial with root $r$.  
Then, $(G; s)\oplus (H; V(H))$  
is an $\alpha$-radial with strong root $s$ whose strong ground is $H$.  
\end{lemma}

Finally, we prove two lemmas for radials with sublinear root $r$.  
Lemma~\ref{lem:triplex2const} is implied from Lemmas~\ref{lem:edgeaddr} and \ref{lem:construct2r}.

\begin{lemma} \label{lem:triplex2const} 
Let $\alpha \in \{+, -\}$.   
Let $H_1$ be an almost strong $\alpha$-radial with root $r$, 
$H_2$ be a sublinear $\alpha$-radial with root $r$ such that $V(H_2)\cap V(H_1) = \{r\}$, 
and let $H_3$ be a linear $\alpha$-semiradial with root $s$ such that $V(H_3)\cap V(H_i) = \emptyset$ for every $i \in \{1, 2\}$.  
Assume that if $V(H_1) = \{r\}$ then $V(H_3) = \{s\}$. 
Let $\hat{H}$ be a gluing sum  $(H_3; s) \oplus (H_1 + H_2; V(H_1)\setminus \{r\})$. 
Let $I$ be $\hat{H}$ or a bidirected graph obtained from $\hat{H}$ by adding some edges between 
\begin{rmenum} 
\item $V(H_2)\setminus \{r\}$ and $( V(H_1)\setminus \{r\}) \cup ( V(H_3)\setminus \{s\})$ in which the end in $V(H_2)$ has sign $\alpha$,  or 
\item $r$ and $V(H_3)\setminus \{s\}$ in which $r$ has sign $\alpha$. 
\end{rmenum} 
Then,  $I$ is a triplex $\alpha$-radial with root $r$  
in which 
$H_1$ is the almost strong ground, 
and $V(H_2) \setminus \{r\}$ and $ V(H_3)\setminus \{s\}$ are the first and second shells, respectively. 
\end{lemma} 
\begin{proof} 
It can easily be confirmed that 
$H_1 + H_2$ is an $\alpha$-radial with root $r$ 
for which $V(H_1)\setminus \{r\} = \reachset{H_1 + H_2}{r}{-\alpha}{-\alpha}$. 
Hence, Lemma~\ref{lem:construct2r} implies that 
$V(H_1)\setminus \{r\} =  \reachset{\hat{H}}{r}{-\alpha}{-\alpha}$, 
and $V(H_2) \cup V(H_3)\setminus \{s\} = \reachset{\hat{H}}{r}{\alpha}{-\alpha}\setminus \reachset{\hat{H}}{r}{-\alpha}{-\alpha}$. 
Therefore, 
Lemma~\ref{lem:edgeaddr} implies  
$\reachset{I}{r}{\beta}{-\alpha} = \reachset{\hat{H}}{r}{\beta}{-\alpha}$ for every $\beta \in \{+, -\}$. 
Thus, the claim follows. 
\end{proof}

Lemma~\ref{lem:const2rrtriplex} is implied from Lemmas~\ref{lem:construct2sublin} and \ref{lem:rradd}.

\begin{lemma} \label{lem:const2rrtriplex} 
Let $\alpha \in \{+, -\}$.  
Let $G$ and $H$ be disjoint bidirected graphs such that 
 $G$ is a round $\alpha$-radial with sublinear root $s$, and 
 $H$ is a triplex $\alpha$-radial with root $r$ whose shell $S$ is nonempty.  
Let $\hat{G}$ be a gluing sum $(G; s) \oplus (H; S)$ 
or a bidirected graph obtained from the gluing sum by adding some edges joining $r$ and $V(G)\setminus \{s\}$ 
in which $r$ has the sign $\alpha$. 
Then, $\hat{G}$ is an $\alpha$-radial with sublinear root $r$ whose extended ground is $H$. 
\end{lemma} 
\begin{proof} 
Let $\hat{G}'$ be $(G; s) \oplus (H; S)$. 
Then, Lemma~\ref{lem:construct2sublin} implies that 
$\hat{G}$ is an $\alpha$-radial with sublinear root $r$ 
in which $H$ is the extended ground and $S$ is its shell. 
Lemma~\ref{lem:rradd} further implies the claim for $\hat{G}$. 
\end{proof}

\section{Characterization} \label{sec:characterization} 
\subsection{Characterizations of Round Radials and Sharp Semiradials} \label{sec:characterization:rrssr}

In Section~\ref{sec:characterization:rrssr},  
we provide constructive characterizations of round radials and sharp semiradials.

\begin{definition} 
Let $\alpha\in\{+, -\}$. 
Two sets $\rrset{r}{\alpha}$ and $\lsrset{r}{\alpha}$ of bidirected graphs with vertex $r$ are defined as follows: 
\begin{rmenum} 
\item Every almost strong $\alpha$-radial with root $r$ is a member of $\rrset{r}{\alpha}$. 
\item Every linear $\alpha$-semiradial with root $r$ is a member of $\lsrset{r}{\alpha}$. 
\item Let $G\in \rrset{r}{\alpha}$. Then, a bidirected graph obtained from $G$ by adding some edges  
of the form $rv$, where $v\in V(G)$, in which  $r$ has sign $\alpha$ is a member of $\rrset{r}{\alpha}$. 
\item Let $G \in \lsrset{s}{\alpha}$, and let $H$ be an almost strong $\alpha$-radial with root $r$ such that $V(G)\cap V(H) = \emptyset$. 
Then, any gluing sum $(G; s) \oplus (H; V(H)\setminus \{r\})$ is a member of $\rrset{r}{\alpha}$. 
\item 
Let $G \in \rrset{s}{\alpha}$,  and let $H$ be a linear $\alpha$-semiradial $H$ with root $r$ such that $V(G)\cap V(H) = \emptyset$. 
Then, any gluing sum $(G; s) \oplus (H; V(H)\setminus \{r\})$ is a member of $\lsrset{r}{\alpha}$. 
\end{rmenum}   
\end{definition}

From Lemmas~\ref{lem:lg2asg}, \ref{lem:asg2lg}, \ref{lem:const2srlg}, and \ref{lem:const2rasg}, 
Theorem~\ref{thm:rrssr} is easily obtained.

\begin{theorem} \label{thm:rrssr} 
Let $\alpha \in \{+, -\}$. 
A bidirected graph with vertex $r$ is a member of $\rrset{r}{\alpha}$ 
if and only if it is a round $\alpha$-radial with root $r$. 
A bidirected graph with vertex $r$ is a member of $\lsrset{r}{\alpha}$ 
if and only if it is a linear $\alpha$-semiradial with root $r$. 
\end{theorem}

\subsection{Characterizations of General Radials and Semiradials} \label{sec:characterization:rsr}

In Section~\ref{sec:characterization:rsr}, 
we finally present constructive characterizations for general radials and semiradials  
using results from Sections~\ref{sec:decomposition}, \ref{sec:construct}, and \ref{sec:characterization:rrssr}. 
Lemmas~\ref{lem:abg2lg} and \ref{lem:const2srag} imply Theorem~\ref{thm:sr}.

\begin{theorem} \label{thm:sr} 
Let $\alpha \in \{+, -\}$.  
A bidirected graph with vertex $r$ is an $\alpha$-semiradial with root $r$ 
if and only if 
it is a gluing sum $(G, s) \oplus (H; V(H))$, 
where $G$ is a member of $\lsrset{s}{\alpha}$ 
and $H$ is an absolute semiradial with root $r$. 
\end{theorem}

Lemmas~\ref{lem:stg2lg} and \ref{lem:const2rstg} imply Theorem~\ref{thm:strr}.

\begin{theorem}  \label{thm:strr} 
Let $\alpha \in \{+, -\}$.  
A bidirected graph with vertex $r$ is an $\alpha$-radial with strong root $r$ 
if and only if 
it is a gluing sum $(G, s) \oplus (H; V(H))$, 
where $G$ is a member of $\lsrset{s}{\alpha}$ 
and $H$ is a strong $\alpha$-radial with root $r$.  
\end{theorem}

Lemmas~\ref{lem:shell} and \ref{lem:triplex2const} imply Theorem~\ref{thm:triplex}.

\begin{theorem} \label{thm:triplex} 
Let $\alpha \in \{+, -\}$.  
The following two statements are equivalent: 
\begin{rmenum} 
\item  A bidirected graph $G$ is an $\alpha$-triplex radial with root $r$. 
\item Let $H_1$ be  an almost strong $\alpha$-radial with root $r$, 
 $H_2$ be a sublinear $\alpha$-radial with root $r$, 
and $H_3$ be a linear $\alpha$-semiradial with root $s$ 
such that $V(H_2)\cap V(H_1) = \{r\}$, $V(H_3)\cap V(H_i) = \emptyset$ for every $i \in \{1, 2\}$, 
and if $V(H_1) = \{r\}$ then $V(H_3) = \{s\}$. 
A bidirected graph $G$ is a 
gluing sum  $(H_3; s) \oplus (H_1 + H_2; V(H_1)\setminus \{r\})$ 
or a bidirected graph obtained from this sum by adding some edges between 
\begin{rmenum} 
\item $V(H_2)\setminus \{r\}$ and $( V(H_1)\setminus \{r\}) \cup ( V(H_3)\setminus \{s\})$ in which the end in $V(H_2)$ has the sign $\alpha$,  or 
\item $r$ and $V(H_3)\setminus \{s\}$ in which $r$ has the sign $\alpha$. 
\end{rmenum} 

\end{rmenum} 
\end{theorem}

Lemmas~\ref{lem:shell2astg} and \ref{lem:const2rrtriplex} imply Theorem~\ref{thm:sublinrr}.

\begin{theorem} \label{thm:sublinrr} 
Let $\alpha \in \{+, -\}$.  
Bidirected graph with vertex $r$ is an $\alpha$-radial with sublinear root $r$ 
if and only if 
it is a triplex $\alpha$-radial with root $r$ or a gluing sum $(G, s) \oplus (H; S)$, 
where $G$ is a member of $\rrset{s}{\alpha}$ 
and $H$ is a triplex $\alpha$-radial with root $r$ whose shell is $S$. 
\end{theorem}

Combined with Theorem~\ref{thm:rrssr}, 
Theorems~\ref{thm:sr} and \ref{thm:strr} provide 
constructive characterizations of 
 semiradials and radials with strong root, respectively.  
Also, Theorems~\ref{thm:triplex} and \ref{thm:sublinrr} provide 
a constructive characterization of radials with sublinear root. 
 Here, 
 absolute semiradials, strong and almost strong radials, linear semiradials, and sublinear radials 
 serve as fundamental building blocks.

\bibliographystyle{splncs03.bst}
\bibliography{sear.bib}

\setcounter{theorem}{0}
\renewcommand{\thetheorem}{A.\arabic{theorem}}

\section*{Appendix} 

In the Appendix, we present some preliminary definitions and results from our preceding paper~\cite{kita2019constructive}.

\begin{definition} 
Let $G$ be a bidirected graph,  let $r\in V(G)$, and let $\alpha\in\{+, -\}$. 
We call $G$ an {\em $\alpha$-radial} with {\rm root} $r$ 
if, for every $v\in V(G)$,  there is an $(\alpha, -\alpha)$-ditrail from $v$ to $r$.  
We call $G$ an {\em $\alpha$-semiradial} with {\rm root} $r$ 
if, for every $v\in V(G)$,  there is an $\alpha$-ditrail from $v$ to $r$.  
\end{definition}

\begin{definition} 
Let $G$ be a bidirected graph, and let $r\in V(G)$. 
We call $G$ an {\em absolute semiradial} with root $r$ 
if $G$ is an $\alpha$-semiradial with root $r$ for each $\alpha\in\{+, -\}$. 
\end{definition} 

\begin{definition} 
Let $\alpha\in\{+, -\}$. 
An $\alpha$-radial $G$ with root $r\in V(G)$ is said to be {\em strong} 
if,  for every $v\in V(G)$,  there is a $(-\alpha, -\alpha)$-ditrail from $v$ to $r$. 
An $\alpha$-radial $G$ with root $r$ is said to be {\em almost strong} 
if, for every $v\in V(G)\setminus \{r\}$, 
there is a $(-\alpha, -\alpha)$-ditrail from $v$ to $r$, 
however there is no closed $(-\alpha, -\alpha)$-ditrail over $r$. 
\end{definition}

\begin{definition} 
Let $\alpha\in\{+, -\}$. 
Let $G$ be an $\alpha$-semiradial with root $r$.  
We say that $G$ is {\em linear} if $G$ has no loop edge over $r$, and $G$ has no $-\alpha$-ditrails from $x$ to $r$
 for any $x\in V(G)$ 
except for the trivial $(-\alpha, \alpha)$-ditrail from $r$ to $r$ with no edge.  
We say that $G$ is {\em sublinear} if if $G$ has  no $(-\alpha, -\alpha)$-ditrails from $x$ to $r$ for any $x\in V(G)$. 
\end{definition}

\begin{definition} 
We define a set $\sqrset{r}$ of bidirected graphs with vertex $r$ as follows: 
\begin{rmenum} 
\item 
The graph that consists of a single vertex $r$ and no edge is a member of $\sqrset{r}$. 
\item 
Let $H\in\sqrset{r}$, and 
let $P$ be a diear relative to $H$. 
Then, $H + P$ is a member of $\sqrset{r}$. 
\end{rmenum} 
\end{definition}

\begin{theorem} \label{thm:asr} 
Let $r$ be a vertex symbol. 
Then, $\sqrset{r}$ is the set of absolute semiradials with root $r$. 
\end{theorem}

\begin{definition} 
Let $r$ be a vertex symbol, and let $\alpha \in \{+, -\}$. 
We define a set $\acset{r}{\alpha}$ of bidirected graphs that have vertex $r$ as follows: 
\begin{rmenum} 
\item \label{item:base} 
A simple $(-\alpha, -\alpha)$-diear  relative to $r$  is an element of $\acset{r}{\alpha}$. 
\item \label{item:inductive}
If $G\in \acset{r}{\alpha}$ holds and $P$ is a diear   relative to $G$, 
then $G + P \in \acset{r}{\alpha}$ holds. 
\end{rmenum}
\end{definition}

\begin{theorem} \label{thm:str} 
Let $r$ be a vertex symbol. 
Then, $\acset{r}{\alpha}$ is the set of strong $\alpha$-radials with root $r$. 
\end{theorem}

\begin{definition} 
Let $\alpha\in\{+, -\}$. 
Define a set $\nacset{r}{\alpha}$ of bidirected graphs with a vertex $r$ as follows: 
\begin{rmenum} 
\item \label{item:const:base} Let $\beta \in \{+, -\}$,  let $r'$ be a vertex symbol distinct from $r$, 
let $G\in \acset{r'}{\beta}$ be a bidirected graph with $r\not\in V(G)$, 
and let $rr'$ be an edge for which the signs of $r$ and $r'$ are $-\alpha$ and $\beta$, respectively. 
Then, $G + rr'$ is a member of $\nacset{r}{\alpha}$. 
\item \label{item:const:edge}
Let $G\in \nacset{r}{\alpha}$, 
let $v\in V(G)$,  
and let $rv$ be an edge in which the sign of $r$ is  $\alpha$. 
Then, $G + rv$ is a member of $\nacset{r}{\alpha}$. 
\item \label{item:const:vertex} 
Let $G_1, G_2\in \nacset{r}{\alpha}$ be bidirected graphs with $V(G_1)\cap V(G_2) = \{r\}$. 
Then, $G_1 + G_2$ is a member of $\nacset{r}{\alpha}$.  
\end{rmenum} 
\end{definition} 

The following theorem is our constructive characterization of almost strong radials. 

\begin{theorem} \label{thm:asr2char} 
Let $r$ be a vertex symbol. 
Let $\alpha \in \{+, -\}$.  
Then, $\nacset{r}{\alpha}$ is the set of almost strong $\alpha$-radials with root $r$. 
\end{theorem}

\end{document}